\documentclass[11pt,leqno]{article}
\usepackage[margin=1in]{geometry} 
\usepackage{amssymb,amsfonts,amsmath,bbm,mathrsfs,stmaryrd,mathtools,mathabx}
\usepackage{xcolor}
\usepackage{url}

\usepackage{extarrows}

\usepackage[shortlabels]{enumitem}
\usepackage{tensor}

\usepackage{xr}
\usepackage[T1]{fontenc}

\usepackage[colorlinks,
linkcolor=black!75!red,
citecolor=blue,
pdftitle={},
pdfproducer={pdfLaTeX},
pdfpagemode=None,
bookmarksopen=true,
bookmarksnumbered=true,
backref=page]{hyperref}

\usepackage{tikz}
\usetikzlibrary{arrows,calc,decorations.pathreplacing,decorations.markings,intersections,shapes.geometric,through,fit,shapes.symbols,positioning,decorations.pathmorphing}

\usepackage{braket}

\usepackage[amsmath,thmmarks,hyperref]{ntheorem}
\usepackage{cleveref}

\newcommand\nthalias[1]{\AddToHook{env/#1/begin}{\crefalias{lemma}{#1}}}

\nthalias{definition}
\nthalias{example}
\nthalias{examples}
\nthalias{remark}
\nthalias{remarks}
\nthalias{convention}
\nthalias{notation}
\nthalias{construction}
\nthalias{theoremN}
\nthalias{propositionN}
\nthalias{corollaryN}
\nthalias{lemma}
\nthalias{proposition}
\nthalias{corollary}
\nthalias{theorem}
\nthalias{conjecture}
\nthalias{question}
\nthalias{assumption}

\creflabelformat{enumi}{#2#1#3}

\crefname{section}{Section}{Sections}
\crefformat{section}{#2Section~#1#3} 
\Crefformat{section}{#2Section~#1#3} 

\crefname{subsection}{\S}{\S\S}
\AtBeginDocument{%
  \crefformat{subsection}{#2\S#1#3}%
  \Crefformat{subsection}{#2\S#1#3}%
}

\crefname{subsubsection}{\S}{\S\S}
\AtBeginDocument{%
  \crefformat{subsubsection}{#2\S#1#3}%
  \Crefformat{subsubsection}{#2\S#1#3}%
}

\theoremstyle{plain}

\newtheorem{lemma}{Lemma}[section]
\newtheorem{proposition}[lemma]{Proposition}
\newtheorem{corollary}[lemma]{Corollary}
\newtheorem{theorem}[lemma]{Theorem}

\newtheorem{question}[lemma]{Question}

\theoremstyle{plain}
\theoremnumbering{Alph}
\newtheorem{theoremN}{Theorem}

\theoremstyle{plain}
\theorembodyfont{\upshape}
\theoremsymbol{\ensuremath{\blacklozenge}}

\newtheorem{definition}[lemma]{Definition}
\newtheorem{example}[lemma]{Example}

\newtheorem{remark}[lemma]{Remark}
\newtheorem{remarks}[lemma]{Remarks}

\crefname{definition}{definition}{definitions}
\crefformat{definition}{#2definition~#1#3} 
\Crefformat{definition}{#2Definition~#1#3} 

\crefname{ex}{example}{examples}
\crefformat{example}{#2example~#1#3} 
\Crefformat{example}{#2Example~#1#3} 

\crefname{exs}{example}{examples}
\crefformat{examples}{#2example~#1#3} 
\Crefformat{examples}{#2Example~#1#3} 

\crefname{remark}{remark}{remarks}
\crefformat{remark}{#2remark~#1#3} 
\Crefformat{remark}{#2Remark~#1#3} 

\crefname{remarks}{remark}{remarks}
\crefformat{remarks}{#2remark~#1#3} 
\Crefformat{remarks}{#2Remark~#1#3} 

\crefname{convention}{convention}{conventions}
\crefformat{convention}{#2convention~#1#3} 
\Crefformat{convention}{#2Convention~#1#3} 

\crefname{notation}{notation}{notations}
\crefformat{notation}{#2notation~#1#3} 
\Crefformat{notation}{#2Notation~#1#3} 

\crefname{table}{table}{tables}
\crefformat{table}{#2table~#1#3} 
\Crefformat{table}{#2Table~#1#3}

\crefname{lemma}{lemma}{lemmas}
\crefformat{lemma}{#2lemma~#1#3} 
\Crefformat{lemma}{#2Lemma~#1#3} 

\crefname{proposition}{proposition}{propositions}
\crefformat{proposition}{#2proposition~#1#3} 
\Crefformat{proposition}{#2Proposition~#1#3} 

\crefname{propositionN}{proposition}{propositions}
\crefformat{propositionN}{#2proposition~#1#3} 
\Crefformat{propositionN}{#2Proposition~#1#3} 

\crefname{corollary}{corollary}{corollaries}
\crefformat{corollary}{#2corollary~#1#3} 
\Crefformat{corollary}{#2Corollary~#1#3} 

\crefname{corollaryN}{corollary}{corollaries}
\crefformat{corollaryN}{#2corollary~#1#3} 
\Crefformat{corollaryN}{#2Corollary~#1#3} 

\crefname{theorem}{theorem}{theorems}
\crefformat{theorem}{#2theorem~#1#3} 
\Crefformat{theorem}{#2Theorem~#1#3} 

\crefname{theoremN}{theorem}{theorems}
\crefformat{theoremN}{#2theorem~#1#3} 
\Crefformat{theoremN}{#2Theorem~#1#3} 

\crefname{enumi}{}{}
\crefformat{enumi}{#2#1#3}
\Crefformat{enumi}{#2#1#3}

\crefname{assumption}{assumption}{Assumptions}
\crefformat{assumption}{#2assumption~#1#3} 
\Crefformat{assumption}{#2Assumption~#1#3} 

\crefname{construction}{construction}{Constructions}
\crefformat{construction}{#2construction~#1#3} 
\Crefformat{construction}{#2Construction~#1#3} 

\crefname{question}{question}{Questions}
\crefformat{question}{#2question~#1#3} 
\Crefformat{question}{#2Question~#1#3} 

\crefname{equation}{}{}
\crefformat{equation}{(#2#1#3)} 
\Crefformat{equation}{(#2#1#3)}

\numberwithin{equation}{section}

\theoremstyle{nonumberplain}
\theoremsymbol{\ensuremath{\blacksquare}}

\newtheorem{proof}{Proof}
\newcommand\pf[1]{\newtheorem{#1}{Proof of \Cref{#1}}}

\newcommand\bC{{\mathbb C}}

\newcommand\bG{{\mathbb G}}

\newcommand\bR{{\mathbb R}}

\newcommand\bZ{{\mathbb Z}}

\newcommand\cA{{\mathcal A}}
\newcommand\cB{{\mathcal B}}
\newcommand\cC{{\mathcal C}}

\newcommand\cE{{\mathcal E}}
\newcommand\cF{{\mathcal F}}

\newcommand\cK{{\mathcal K}}

\newcommand\cS{{\mathcal S}}

\DeclareMathOperator{\id}{id}

\DeclareMathOperator{\End}{\mathrm{End}}

\DeclareMathOperator{\Aut}{\mathrm{Aut}}

\newcommand{\cat}[1]{\textsc{#1}}

\newcommand{\qedhere}{\mbox{}\hfill\ensuremath{\blacksquare}}

\newcommand{\xrightarrowdbl}[2][]{%
  \xrightarrow[#1]{#2}\mathrel{\mkern-14mu}\rightarrow
}

\title{Non-commutative branched covers and bundle unitarizability}
\author{Alexandru Chirvasitu}

\begin{document}

\date{}

\newcommand{\Addresses}{{
  \bigskip
  \footnotesize

  \textsc{Department of Mathematics, University at Buffalo}
  \par\nopagebreak
  \textsc{Buffalo, NY 14260-2900, USA}  
  \par\nopagebreak
  \textit{E-mail address}: \texttt{achirvas@buffalo.edu}


}}

\maketitle

\begin{abstract}
  We prove that (a) the sections space of a continuous unital subhomogeneous $C^*$ bundle over compact metrizable $X$ admits a finite-index expectation onto $C(X)$, answering a question of Blanchard-Gogi\'{c} (in the metrizable case); (b) such expectations cannot, generally, have ``optimal index'', answering negatively a variant of the same question; and (c) a homogeneous continuous Banach bundle over a locally paracompact base space $X$ can be renormed into a Hilbert bundle in such a manner that the original space of bounded sections is $C_b(X)$-linearly Banach-Mazur-close to the resulting Hilbert module over the algebra $C_b(X)$ of continuous bounded functions on $X$. This last result resolves quantitatively another problem posed by Gogi\'{c}. 
\end{abstract}

\noindent {\em Key words:
  Banach bundle;
  Banach-Mazur distance;
  L\"owner-John ellipsoid;
  Vietoris topology;
  expectation;
  finite index;
  locally paracompact;
  locally trivial

}

\vspace{.5cm}


\tableofcontents

\section*{Introduction}

We are concerned here with {\it Banach bundles} in the sense of \cite[Definition 1.1]{dg_ban-bdl} (conventions on the precise meaning of the phrase sometimes vary in the literature): 
\begin{itemize}
\item Continuous open maps
  \begin{equation*}
    \text{{\it total space} }
    \cE
    \xrightarrowdbl{\quad\pi\quad}
    X
    \text{ ({\it base space})};
  \end{equation*}

\item with Banach-space structures on the {\it fibers} $\cE_x:=\pi^{-1}(x)$, $x\in X$;

\item so that addition and scalar multiplication are continuous in the guessable sense, as maps
  \begin{equation*}
    \left(\text{scalar field }\bR\text{ or }\bC\right)\times \cE
    \xrightarrow{\quad}
    \cE
    \quad\text{and}\quad
    \left\{(p,q)\in \cE^2\ |\ \pi(p)=\pi(q)\right\}
    =:
    \cE\times_X\cE
    \xrightarrow{\quad}
    \cE;
  \end{equation*}

\item the norm $\cE\xrightarrow{\|\cdot\|}\bR_{\ge 0}$ is either
  \begin{itemize}
  \item continuous ({\it continuous} or {\it (F) bundles}, for {\it Fell} \cite[\S 1]{fell_ext}; overwhelmingly the focus of the paper);
  \item or {\it upper semicontinuous} \cite[Problem 7K]{wil_top} ({\it (H) Bundles}, for {\it Hofmann} \cite[Definitions 3.2 and 3.3]{zbMATH03539905});
  \end{itemize}
\item and such that
  \begin{equation*}
    \cE_{U,<\varepsilon}
    :=
    \left\{p\in \cE\ |\ \pi(p)\in U\quad\text{and}\quad \|p\|<\varepsilon\right\}
    ,\quad \text{nbhd }U\ni x\in X
    ,\quad \varepsilon>0
  \end{equation*}
  form a fundamental system of neighborhoods of the zero element $0_x\in \cE_x$. 
\end{itemize}
The terminology extends in the obvious fashion to  \cite[p.9]{dg_ban-bdl} to bundles of Hilbert spaces, Banach or $*$- or $C^*$-algebras, etc. 

The present paper was originally motivated by problems mentioned and in part addressed in \cite{MR3056657,bg_cx-exp}, revolving around the common theme of conveniently renorming continuous Banach (or $C^*$) bundles into Hilbert bundles (this is what the title's  unitarizability refers to).

There are several ways to make sense of this, with a common general flavor. The primary motivating interest in \cite{bg_cx-exp}, for instance, is the theory of non-commutative (or quantum) {\it branched covers}, initiated in \cite{pt_brnch} (whose introduction defines the phrase, classically, as a continuous, open, bounded-fiber-cardinality surjection between compact Hausdorff spaces with). That paper's main result \cite[Theorem 1.1]{pt_brnch} states that a surjection $Y\to X$ between compact Hausdorff spaces is a branched cover in this sense if and only if the corresponding embedding $C(X)\lhook\joinrel\to C(Y)$ admits a finite-index conditional expectation. A brief recollection:
\begin{itemize}[wide]
\item A {\it conditional expectation} \cite[Definition II.6.10.1 and Theorem II.6.10.2]{blk} $A\xrightarrow{E}B$ of a $C^*$-algebra onto a $C^*$-subalgebra is an idempotent map of norm 1.

  \cite[Theorem II.6.10.2]{blk} ensures the equivalence between that characterization and the seemingly stronger \cite[Definition II.6.10.1]{blk}: norm-1 idempotent maps are automatically {\it completely positive} \cite[Definition II.6.9.1]{blk} and $B$-bimodule maps. 
 
\item A conditional expectation $A\xrightarrow{E}B$ has {\it finite index} \cite[Definition 2]{fk_fin-ind} if
  \begin{equation*}
    K(E):=\inf\{K\ge 1\ |\ K\cdot E-\id \ge 0\text{ (i.e. the map is positive)}\}<\infty.
  \end{equation*}
  We will refer to $K(E)$ as the {\it $K$-constant of $E$}.
\end{itemize}
The aforementioned classical result then motivates \cite[Definition 1.2]{pt_brnch}, whereby non-commutative branched covers are unital $C^*$ embeddings admitting finite-index expectations. With this in place, \cite[Problem 3.11]{bg_cx-exp} asks:

\begin{question}\label{qu:bg_pb3.11}
  Suppose $\cA\xrightarrowdbl{}X$ is a {\it subhomogeneous} (i.e. \cite[Definition IV.1.4.1]{blk} $\sup_x \dim\cA_x<\infty$) unital (F) $C^*$-algebra bundle over compact Hausdorff $X$.
  \begin{enumerate}[(1),wide]
  \item\label{item:qu:bg_pb3.11:fin.ind} Does the embedding
    \begin{equation*}
      C(X)
      \lhook\joinrel\xrightarrow{\quad}
      \Gamma(\cA)
      :=
      \left\{
        \text{(continuous) {\it sections} \cite[p.9]{dg_ban-bdl} of $\cA$}
      \right\}
    \end{equation*}
    admit a finite-index expectation?
    
  \item\label{item:qu:bg_pb3.11:opt.ind} Having defined the {\it rank} \cite[Definition 2.4]{bg_cx-exp} $r(\cA)$ as
    \begin{equation*}
      r(\cA):=\sup_{x\in X}\left\{\sum\left(\text{dimensions of irreducible $\cA_x$-representations}\right)\right\},
    \end{equation*}  
    does the embedding $C(X)\lhook\joinrel\to A:=\Gamma(\cA)$ admit a conditional expectation $A\xrightarrow{E}C(X)$ with (the optimal value \cite[Theorem 1.4]{bg_cx-exp}) $K(E)=r(\cA)$?  
  \end{enumerate}
\end{question}

One of the goals of the paper is to address the first part of \Cref{qu:bg_pb3.11} affirmatively over metrizable $X$ (\Cref{th:nc.brnch.cov.cpctmetr}), and the second part negatively in \Cref{ex:bg_pb3.11_no} as well as \Cref{le:plbk.cones} and \Cref{pr:same.brat}. The bound $r(\cA)\ge 3$ is the best one can expect in the second part of the following statement, by \cite[Proposition 3.7]{bg_cx-exp}.

\begin{theoremN}\label{thn:nc.brnch.cov}
  Let $\cA\xrightarrowdbl{}X$ be a unital subhomogeneous (F) $C^*$ bundle over a compact metrizable space.

  \begin{enumerate}[(1),wide]
  \item The embedding
    \begin{equation*}
      \left\{
        \text{bounded continuous functions on $X$}
      \right\}
      =:
      C_b(X)
      \lhook\joinrel\xrightarrow{\quad}
      \Gamma_b(\cA)
      :=
      \left\{
        \text{bounded sections of $\cA$}
      \right\}
    \end{equation*}
    has a finite-index expectation.

  \item That expectation cannot always be chosen to have $K$-constant $r(\cA)$: counterexamples exist as soon as $r(\cA)\ge 3$.  \qedhere
  \end{enumerate}  
\end{theoremN}

The expectations of \Cref{thn:nc.brnch.cov} will equip $\Gamma(\cA)$ with a {\it Hilbert $C(X)$-module} \cite[Definition 15.1.5]{wo} structure
\begin{equation*}
  \Gamma(\cA)^2\ni
  (s,t)
  \xmapsto{\quad\braket{-\mid -}\quad}
  E(s^*t)
  \in
  C(X)
  :=\text{bounded continuous functions on $X$}
\end{equation*}
and will correspondingly make the fibers $\cA_x$ into Hilbert spaces, in such a fashion that the {\it Banach-Mazur distance} (recalled in \Cref{eq:bm.orig}) between the induced Hilbert-space norm $\|\cdot\|_{2,x}$ at $x$ and the original Banach norm $\|\cdot\|_x$ is bounded in $x\in X$. This justifies bringing up (well-behaved) unitarizability. 

The focus is somewhat different in \cite{MR3056657}, but some of the same questions present themselves. \cite[Problem 4.9]{MR3056657} is more or less \Cref{qu:bg_pb3.11}\Cref{item:qu:bg_pb3.11:fin.ind}, with the inessential caveat of working over {\it locally} compact base spaces and spaces $\Gamma_0(\cA)$ of sections vanishing at infinity (the usual {\it unitization} \cite[\S II.1.2]{blk} gadgetry easily bridges the gap). \cite[Problem 4.7]{MR3056657} transports the unitarizability issue to Banach (rather than $C^*$-algebra) bundles, the connection being that outlined in the preceding paragraph:

\begin{question}\label{qu:gog.prb.4.7}
  Let $\cE\xrightarrowdbl{} X$ be an {\it $n$-homogeneous} (F) Banach bundle over a locally compact $T_2$ space, i.e. \cite[Definition 2.2]{dupre_hilbund-1} one whose fibers are all $n$-dimensional for some $n\in \bZ_{\ge 0}$.

  Does the space $\Gamma_0(\cE)$ of sections vanishing at $\infty$ admit a {\it $C_0(X)$-valued inner product} \cite[Definition 15.1.1]{wo} 
  \begin{equation*}
    \Gamma_0(E)^2
    \xrightarrow[\quad]{\quad\braket{-\mid-}\quad}
    C_0(X)
  \end{equation*}
  making it into a {\it Hilbert $C_0(X)$-module} \cite[Definition 15.1.5]{wo}, and such that the corresponding norm
  \begin{equation*}
    \Gamma_0(\cE)
    \ni
    s
    \xmapsto{\quad}
    \|\braket{s\mid s}\|^{1/2}_{C_0(X)}
    \in \bR_{\ge 0}
  \end{equation*}
  is equivalent with the original supremum norm on $\Gamma_0(\cE)$? 
\end{question}

\Cref{th:isbddunit} below answers in the affirmative, with a couple of amplifications:
\begin{itemize}[wide]
\item the space can be locally {\it para}compact (Hausdorff), in the sense \cite[Definition 3.1]{gierz_bdls} that every point has a closed {\it paracompact} \cite[Definition 20.6]{wil_top} neighborhood;

\item and there is a quantitative aspect to the problem, in giving the expected upper bound on (a Banach $C_b(X)$-module version: \Cref{def:bm.dist.mod}) of the already-mentioned Banach-Mazur distance between the original $\Gamma_b$ and the analogous space for the Hilbert renorming. 
\end{itemize}

\begin{theoremN}
  An $n$-homogeneous Banach bundle $\cE\xrightarrowdbl{}X$ be over a locally paracompact $T_2$ base space can be equipped with a Hilbert-bundle renorming in such a fashion that
  \begin{equation*}
    \|T\|\cdot \|T^{-1}\|\le \sqrt n
  \end{equation*}
  for
  \begin{equation*}
    \begin{tikzpicture}[>=stealth,auto,baseline=(current  bounding  box.center)]
      \path[anchor=base] 
      (0,0) node (l) {$(\Gamma_b(\cE),\ \text{original norm})$}
      +(8,0) node (r) {$(\Gamma_b(\cE),\ \text{new norm})$}
      ;
      \draw[->] (l) to[bend left=6] node[pos=.5,auto] {$\scriptstyle T:=\id$} (r);
      \draw[->] (r) to[bend left=6] node[pos=.5,auto] {$\scriptstyle T^{-1}$} (l);
    \end{tikzpicture}
  \end{equation*}
  \qedhere
\end{theoremN}

\subsection*{Acknowledgements}

I am grateful for very fruitful exchanges with I. Gogi{\'c}, on the contents of \cite{MR3056657,bg_cx-exp} and much more.  


\section{Equivalent Hilbert-bundle structures}\label{se:ban2hilb}

For the most part, the choice of base field ($\bR$ or $\bC$) will make no difference; the discussion will thus be mostly field-independent, with appropriate caveats mentioned occasionally along the way. Meanings of symbols might change appropriately: $C_b(X)$ ($C_0(X)$ for locally compact $X$), for instance, denote algebras of bounded (respectively infinitely-vanishing) continuous functions on a space $X$ valued in the ground field: real or complex, as the case may be. We occasionally drop those subscripts (`0' or `b') when the space is compact Hausdorff. 

One goal, in the current section, is to eventually give an affirmative answer to \cite[Problem 4.7]{MR3056657}, recalled as \Cref{qu:gog.prb.4.7} above. We first recast the problem in terms less specific to LCH spaces: $\Gamma_0$ (sections vanishing at infinity) is only appropriate for work with locally compact spaces, whereas local {\it para}compactness is often sufficient. Additionally, streamlining the terminology will also be convenient.

\begin{definition}\label{def:untrzbl.types}
  Let $\cE\xrightarrowdbl{} X$ be a Banach bundle over a space $X$.
  \begin{enumerate}[(1),wide]
  \item\label{item:def:untrzbl.types:plain} The bundle is {\it unitarizable}
    if it is {\it equivalent} to a Hilbert bundle $\cE'\xrightarrowdbl{} X$, i.e. there is a horizontal homeomorphism, linear along each fiber, making
    \begin{equation}\label{eq:renorm}
      \begin{tikzpicture}[>=stealth,auto,baseline=(current  bounding  box.center)]
        \path[anchor=base] 
        (0,0) node (l) {$\cE$}
        +(2,-.5) node (d) {$X$}
        +(4,0) node (r) {$\cE'$}
        ;
        \draw[->] (l) to[bend left=6] node[pos=.5,auto] {$\scriptstyle \cong$} (r);
        \draw[->>] (r) to[bend left=6] node[pos=.5,auto] {$\scriptstyle $} (d);
        \draw[->>] (l) to[bend right=6] node[pos=.5,auto,swap] {$\scriptstyle $} (d);
      \end{tikzpicture}
    \end{equation}
    commute. We also refer to an equivalence in this sense as a {\it renorming}, so that unitarizability might be phrased as the existence of a Hilbert-bundle renorming. 
    
  \item\label{item:def:untrzbl.types:bdd} It is {\it boundedly (or $\Gamma_b$-)unitarizable} if it admits a Hilbert-bundle renorming so that the identity identifies the respective spaces $\Gamma_b$ of bounded sections.    

    Or: denoting by $\Gamma_{b,\|\cdot\|}$ and $\Gamma_{b,\vvvert\cdot\vvvert}$ the spaces of bounded sections for the two supremum norms $\|\cdot\|$ (original Banach bundle) and $\vvvert\cdot\vvvert$ (renormed Hilbert bundle),
    \begin{equation}\label{eq:id.norm2norm}
      \Gamma_{b,\|\cdot\|}
      \xrightarrow{\quad\id\quad}
      \Gamma_{b,\vvvert\cdot\vvvert}
    \end{equation}
    is a topological isomorphism. 
    
  \item\label{item:def:untrzbl.types:0bdd} Similarly, for locally compact $X$ the bundle is {\it 0-boundedly (or $\Gamma_0$-)unitarizable} if the analogous conclusion holds with spaces $\Gamma_0$ of sections vanishing at infinity in place of $\Gamma_b$.

    In other words, the 0-boundedly unitarizable bundles are exactly those for which \Cref{qu:gog.prb.4.7} has an affirmative answer. 
  \end{enumerate}  
  In all cases, we make the convention that if the original bundle was (F) (which is overwhelmingly the case throughout the present paper), then the equivalent Hilbert-bundle structure is also required to be (F). In other words, the renorming is required to preserve continuity, if present. 
\end{definition}

As far as {\it affirmative} answers to \Cref{qu:gog.prb.4.7} go, `boundedly' is better than `0-boundedly' (which is just as well: the former applies to wider classes of spaces, which was the motivation behind the present aside to begin with).

\begin{lemma}\label{le:bdd>0bdd}
  For an LCH space $X$, boundedly unitarizable (F) Banach bundles are also 0-boundedly so. 
\end{lemma}
\begin{proof}
  Simply observe that when the base space is locally compact, the identity \Cref{eq:id.norm2norm}, assumed a topological isomorphism, maps the subspace
  \begin{equation*}
    \Gamma_{0,\|\cdot\|}
    \le
    \Gamma_{b,\|\cdot\|}
  \end{equation*}
  of sections vanishing at infinity onto its analogue $\Gamma_{0,\vvvert\cdot\vvvert}$. Indeed, $\Gamma_{0,\|\cdot\|}$ is the norm-closure of its dense subspace
  \begin{equation*}
    \begin{aligned}
      \Gamma_{00,\|\cdot\|}
      &:=
        \left\{\text{compactly-supported sections}\right\}\\
      &=
        C_{00}(X)\cdot \Gamma_{0,\|\cdot\|}
        \le
        \Gamma_{0,\|\cdot\|},\\
      C_{00}(X)
      &:=
        \left\{\text{compactly-supported functions}\right\},\\
    \end{aligned}    
  \end{equation*}
  that space is mapped by the $C_0(X)$-morphism \Cref{eq:id.norm2norm} into its analogue $\Gamma_{00,\vvvert\cdot\vvvert}$, and all of this applies to the inverse of \Cref{eq:id.norm2norm}. 
\end{proof}

The following preliminary remark will not, strictly speaking, be needed in order to address \Cref{qu:gog.prb.4.7}, but should be regarded as a cognate of sorts (cf. \Cref{cor:isunitrzbl}). {\it Banach manifolds} are as in \cite[\S 3]{upm_ban}, \cite[\S 5]{bourb_vars_1-7}, etc.

\begin{proposition}\label{pr:locpc.base.plbk}
  Homogeneous (F) Banach bundles over locally paracompact $T_2$ base spaces are pullbacks of locally trivial Banach bundles on completely metrizable analytic Banach manifolds. 
\end{proposition}
\begin{proof}
  Consider a Banach bundle $\cE\xrightarrowdbl{}X$ as in the statement. It is, by \cite[Theorem 3.2]{gierz_bdls}, {\it full} in the sense \cite[Definition 2.1]{gierz_bdls} that arbitrary elements of the total space $\cE$ lie on global sections. Such sections can moreover be chosen bounded, as one can easily correct for unboundedness by rescaling (much as in the proof of \cite[Proposition 2.2]{gierz_bdls}, say): given a section $s$ with $s(x_0)=p\in \cE$, consider a continuous function $X\xrightarrow{\varphi}[0,1]$ with $\varphi(x_0)=1$, and substitute
  \begin{equation*}
    s\rightsquigarrow \left(x\xmapsto{\quad}s(x)\cdot\min(1,\varphi(x))\right).
  \end{equation*}
  We thus have an identification
  \begin{equation}\label{eq:fib.is.quot}
    \Gamma_b(X)/\Gamma_{b\mid x\to 0}(X)\cong \cE_x,
  \end{equation}
  with
  \begin{equation*}
    \Gamma_b:=\text{bounded sections}
    \quad\text{and}\quad
    \Gamma_{b\mid x\to 0}
    :=
    \left\{s\in \Gamma_b\ |\ s(x)=0\right\}.
  \end{equation*}
  The left-hand side of \Cref{eq:fib.is.quot} is a dimension-$n$ quotient space of the Banach space $\Gamma:=\Gamma_b(\cE)$ for $n:=\dim\cE_x$, $x\in X$, and we thus obtain a continuous map
  \begin{equation*}
    X
    \ni x
    \xmapsto{\quad\psi\quad}
    \cE_x
    \cong
    \Gamma/\Gamma_{x\to 0}
    \in
    \bG(\Gamma\xrightarrowdbl{}n)
  \end{equation*}
  with the latter symbol denoting the {\it Grassmannian} (analytic Banach) manifold of $n$-dimensional quotients of $\Gamma$ (see e.g. \cite[Example 3.11]{upm_ban} or \cite[\S 5.2.6]{bourb_vars_1-7}). 

  $\cE$ can then be recovered as the pullback along $\psi$ of the {\it canonical quotient bundle} \cite[\S 3.2.3]{3264} on $\bG(\Gamma,n)$ whose fiber over an $n$-dimensional quotient $\Gamma\xrightarrowdbl{}\Gamma'$ is the selfsame $\Gamma'$. That bundle is locally trivial, as easily seen by trivializing it over an open cover for the Grassmannian realizing the latter as an analytic manifold.

  As for complete metrizability, simply observe that duality identifies the Grassmannian $\bG(\Gamma\xrightarrowdbl{}n)$ with the analogous Grassmannian $\bG(n\lhook\joinrel\to\Gamma^*)$ of $n$-dimensional {\it subspaces} of the dual Banach space $\Gamma^*$, and the topology on $\bG(n\lhook\joinrel\to\Gamma^*)$ is induced by the (complete \cite[Proposition 7.3.7]{bbi}) {\it Hausdorff metric} \cite[Definition 7.3.1]{bbi}
  \begin{equation*}
    d_H(\bullet,\bullet')
    :=
    \max\left(
      \sup_{x\in \bullet}\inf_{x'\in\bullet'}d(x,x')
      ,\quad
      \sup_{x'\in \bullet'}\inf_{x\in\bullet}d(x,x')
    \right)
    ,\quad
    \bullet,\bullet'\subseteq \text{ambient metric space}
  \end{equation*}
  between the unit balls of the $n$-dimensional subspaces of $\Gamma^*$. 
\end{proof}

\begin{remark}\label{re:loctriv.over.locparaco}
  \Cref{pr:locpc.base.plbk} in particular recovers by somewhat unusual means the well-known local triviality of homogeneous (F) bundles over locally {\it para}compact base spaces; the result is often stated assuming local {\it compactness} instead (e.g. \cite[Theorem 18.5]{gierz_bdls}, \cite[\S II.17, Exercise 37]{fd_bdl-1}). \cite[Proposition 2.3]{dupre_hilbund-1} does work with paracompact base spaces, but refers to Hilbert bundles instead; the local triviality argument does, however, easily adapt to Banach bundles.
\end{remark}

\Cref{pr:locpc.base.plbk} will suffice for the following weaker analogue of the sought-after affirmative answer to \Cref{qu:gog.prb.4.7}:

\begin{corollary}\label{cor:isunitrzbl}
  Homogeneous (F) Banach bundles over locally paracompact $T_2$ base spaces are unitarizable. 
\end{corollary}
\begin{proof}
  Recall that the {\it pullbacks} \cite[Chapter 3, construction (b)]{ms-cc} familiar from vector-bundle theory go through fine for Banach bundles (e.g. \cite[pp.22-23]{dg_ban-bdl}): given a categorical pullback \cite[Definition 11.8]{ahs}
\begin{equation*}
  \begin{tikzpicture}[>=stealth,auto,baseline=(current  bounding  box.center)]
    \path[anchor=base] 
    (0,0) node (l) {$\cF:=f^*\cE$}
    +(2,.5) node (u) {$\cE$}
    +(2,-.5) node (d) {$Y$}
    +(4,0) node (r) {$X$}
    ;
    \draw[->] (l) to[bend left=6] node[pos=.5,auto] {$\scriptstyle $} (u);
    \draw[->>] (u) to[bend left=6] node[pos=.5,auto] {$\scriptstyle \text{Banach bundle, (F) or (H)}$} (r);
    \draw[->>] (l) to[bend right=6] node[pos=.5,auto,swap] {$\scriptstyle $} (d);
    \draw[->] (d) to[bend right=6] node[pos=.5,auto,swap] {$\scriptstyle \text{continuous}$} (r);
  \end{tikzpicture}
\end{equation*}
of topological spaces, the bottom left map is again an (F) or respectively (H) Banach bundle. The same goes for Hilbert-bundle structures (they pull back), and the conclusion follows from \Cref{pr:locpc.base.plbk} together with the well-known fact \cite[Theorem 3.9.5, via Theorem 3.5.5]{hus_fib} that (locally trivial) vector bundles over paracompact (e.g. metrizable) base spaces are unitarizable.
\end{proof}

\begin{remarks}\label{res:badmflds}
  \begin{enumerate}[(1),wide]
  \item It is perhaps worth emphasizing that the unitarizability recorded in \Cref{cor:isunitrzbl} is specific to {\it Banach} bundles, as opposed to plain vector bundles: the latter exhibit all manner of pathologies when the base space is not paracompact, such as non-metrizability or failure of the bundle to be trivial even when the space is contractible. This is familiar from the study of non-metrizable (equivalently \cite[Theorem 2.1]{gauld_non-met-mfld}, non-paracompact) manifolds: per \cite[Volume 1, Appendix A to chapters 7, 9, 10, remark 4.]{spiv_dg_3e_1999}, {\it no} non-metrizable smooth manifold has unitarizable tangent bundle.

  \item Following up on the preceding observation, it should be clear how the aforementioned non-unitarizable tangent bundles are {\it not} counterexamples to \Cref{cor:isunitrzbl}: those tangent bundles also fail to carry (F) Banach-bundle structures.

  One need not appeal to \Cref{cor:isunitrzbl} to see this: a Banach-bundle structure on the tangent bundle $TM$ of a smooth (finite-dimensional) manifold $M$ {\it almost} defines a {\it (reversible) Finsler structure} on $M$ (\cite[Definition ]{shen_lec-fin_2001}, \cite[\S 1.1]{bcs_finsler_2000}, \cite[Volume 2, Chapter 4, Addendum on Finsler metrics]{spiv_dg_3e_1999}, etc.). `Almost' because Finsler structures are typically assumed smooth on the {\it slit tangent bundle} \cite[\S 1.1, p.2]{bcs_finsler_2000} $TM\setminus\left(\text{zero section}\right)$. Even assuming only continuity for the norm, as we are, one could define the asymmetric binary function
  \begin{equation*}
    d(x,y):=\inf_{\gamma} L(\gamma)
    ,\quad
    L(\gamma):=\int_{a}^b \|\dot{\gamma}(t)\|\ \mathrm{d}t
  \end{equation*}
  for piecewise smooth curves $\gamma$ connecting $\gamma(a)=x$ and $\gamma(b)=y$, and prove in the same fashion (as in \cite[Lemma 6.2.1]{bcs_finsler_2000}) that it is an asymmetric metric recovering the topology of $M$ (cf. \cite[\S 6.2 C]{bcs_finsler_2000}). Despite asymmetry, one has an estimate of the form 
  \begin{equation*}
    \exists C>0
    \quad:\quad
    \frac 1C d(y,x)\le d(x,y)\le C d(y,x)
    ,\quad
    \forall
    x,y\in M,
  \end{equation*}
  so that
  \begin{equation*}
    D(x,y):=\max\left(d(x,y),\ d(y,x)\right)
  \end{equation*}
  is a genuine metric topologizing $M$. In fact, by the same argument, not only can the tangent bundle of a non-metrizable manifold not be made into a Banach bundle, but will not even admit a weaker structure equipping the fibers with only {\it positively}-homogeneous ``norms'':
  \begin{equation*}
    \|\lambda p\| = \lambda\|p\| = |\lambda|\|p\|
    ,\quad \forall v\in \cE_x
    ,\quad \forall \lambda\in \bR_{>0}
  \end{equation*}
  only, rather than arbitrary scalars $\lambda$. 
  \end{enumerate}
\end{remarks}

To state the strengthening of \Cref{cor:isunitrzbl} alluded to above, we need a measure of ``how boundedly distinct'' two Banach bundles may be. It is a version of the familiar {\it Banach-Mazur distance} (\cite[\S 37]{tj_bm}, \cite[\S 2]{MR1793468}, etc.)
\begin{equation}\label{eq:bm.orig}
  d_{\cat{BM}}\big((E,\|\cdot\|),\ (E',\|\cdot\|')\big)
  :=
  \log\inf\left\{\|T\|\cdot \|T^{-1}\|\ :\ E\xrightarrow[\text{top-linear isomorphism}]{T}E'\right\}
\end{equation}
between Banach spaces, transported over to Banach $C_b(X)$-modules.

\begin{definition}\label{def:bm.dist.mod}
  Let $\cE,\cE'\xrightarrowdbl{}X$ be two Banach bundles over the same base space.
  \begin{enumerate}[(1),wide]
  \item The {\it bounded Banach-Mazur distance} between the two is
  \begin{equation*}
    d_{b\cat{BM}}(\cE,\cE')
    :=
    \log\inf_{T}\|T\|\cdot \|T^{-1}\|
  \end{equation*}
  for
  \begin{equation*}
    \Gamma_{b}(\cE)
    \xrightarrow[\quad\text{Banach $C_b(X)$-module isomorphism}\quad]{T}
    \Gamma_{b}(\cE')
  \end{equation*}
  induced by renormings \Cref{eq:renorm}.

\item In similar fashion, for locally compact Hausdorff $X$ we define the {\it 0-bounded Banach-Mazur distance} $d_{0\cat{BM}}$ with $\Gamma_0$ in place of $\Gamma_b$ throughout. 
  \end{enumerate}
\end{definition}

\begin{remark}
  The proof of \Cref{le:bdd>0bdd} makes it clear that $d_{b\cat{BM}}\ge d_{0\cat{BM}}$, so upper bounds on the former quantity impose upper bounds on the latter. 
\end{remark}

Finally, we have the following quantified improvement on \Cref{cor:isunitrzbl}. 

\begin{theorem}\label{th:isbddunit}
  For an $n$-homogeneous Banach bundle $\cE\xrightarrowdbl{}X$ over a locally paracompact $T_2$ base space we have
  \begin{equation*}
    \inf\left\{d_{b\cat{BM}}(\cE,\ \cE')\ |\ \text{$\cE'\xrightarrowdbl{}X$ Hilbert $n$-homogeneous}\right\}
    \le
    \frac 12 \log n.
  \end{equation*}  
\end{theorem}

In particular:

\begin{corollary}\label{cor:isbddunit}
  Homogeneous (F) Banach bundles over locally paracompact base spaces are boundedly unitarizable in the sense of \Cref{def:untrzbl.types}\Cref{item:def:untrzbl.types:bdd}.  \qedhere
\end{corollary}

And also, piecing together \Cref{cor:isbddunit} and the earlier \Cref{le:bdd>0bdd} into the motivating consequence:

\begin{corollary}\label{cor:is0bddunit}
  Homogeneous (F) Banach bundles over locally compact base spaces are 0-boundedly unitarizable in the sense of \Cref{def:untrzbl.types}\Cref{item:def:untrzbl.types:0bdd}.  \qedhere
\end{corollary}


\pf{th:isbddunit}
\begin{th:isbddunit}
  The respective unit balls $\cE_{x,1}\subset \cE_x$, $x\in X$ of the fibers are origin-symmetric {\it convex bodies} \cite[p.8]{schn_cvx_2e_2014} (i.e. compact convex non-empty sets) with non-empty interior. As such, they have uniquely-determined {\it (L\"owner-)John ellipsoids} $K^L(\cE_{x,1})$ \cite[\S 10.12]{schn_cvx_2e_2014} (also \cite[\S 15]{tj_bm}, which covers the complex case as well): unique minimal-volume ellipsoids containing them.

  We relegate to \Cref{pr:contellips} below the fact that the map
  \begin{equation*}
    X\ni x    
    \xmapsto{\quad}
    K^L(\cE_{x,1})
    \in
    \cat{Cl}(\cE)
    :=
    \left\{\text{closed subsets of $\cE$}\right\}
  \end{equation*}
  is continuous for the {\it Vietoris topology} \cite[\S 1.2]{ct_vietoris} on $\cat{Cl}(\cE)$, by definition having 
  \begin{equation*}
    \begin{aligned}
      U^+
      &:=
        \left\{F\in\cat{Cl}(\cE)\ |\ F\cap U\ne \emptyset\right\}
      \ \text{and}\\
      U^-
      &:=
        \left\{F\in\cat{Cl}(\cE)\ |\ F\subseteq U\right\}
    \end{aligned}
    ,\quad\text{open }U\subseteq \cE
  \end{equation*}
  as sub-basic open sets.

  The $K^L(\cE_{x,1})$ are thus the unit balls of Hilbert-space norms on the fibers $\cE_x$ induced by an (F) Hilbert bundle structure (different, generally, from the one we assumed given initially). We now have the uniform (in $x\in X$) estimates
  \begin{equation*}
    \begin{tikzpicture}[>=stealth,auto,baseline=(current  bounding  box.center)]
      \path[anchor=base] 
      (0,0) node (l) {$\cE_{x,1}$}
      +(8,0) node (r) {$K^L(\cE_{x,1})$}
      ;
      \draw[->] (l) to[bend left=6] node[pos=.5,auto] {$\scriptstyle \|\id\|\le 1\text{: obvious}$} (r);
      \draw[->] (r) to[bend left=6] node[pos=.5,auto] {$\scriptstyle \|\id\|\le {\sqrt{n}}\text{: \cite[Theorem 15.5, equation (15.9)]{tj_bm}}$} (l);
    \end{tikzpicture}
  \end{equation*}
  and the identity on $\cE$ is indeed a topological and linear isomorphism between the original Banach module and the newly-defined Hilbert $C_b(X)$-module.
\end{th:isbddunit}

\begin{remarks}\label{res:l.vs.j}
  \begin{enumerate}[(1),wide]
  \item The notation $K^L$ in the proof of \Cref{th:isbddunit} follows the convention adopted in \cite[\S 10.12]{schn_cvx_2e_2014}, where the (unique) minimal-volume ellipsoid containing a given convex body is named after L\"owner, whereas its dual counterpart, the maximal-volume ellipsoid contained therein (the $K_J$ of \Cref{pr:contellips} below), is named after John.

    As noted there, this is not necessarily historically accurate (\cite{zbMATH06155083} covers some of that history), and authors might differ on terminology.

  \item Ellipsoids can be defined, once a inner product {\it has} been fixed, as the images of the unit ball through {\it affine} transformations
    \begin{equation*}
      \bR^n\ (\bC^n)
      \xrightarrow{\quad(\text{translation})\circ(\text{invertible linear})\quad}
      \bR^n\ (\text{respectively }\bC^n).
    \end{equation*}
    One need not, however, have fixed an inner product beforehand, as ellipsoids can also be defined intrinsically any number of ways; the origin-symmetric ones, for instance (from which all others can be obtained by translation) are
    \begin{itemize}[wide]
    \item the sets
      \begin{equation*}
        \left\{\sum \lambda_i v_i\ \big|\ \sum|\lambda_i|^2\le 1\right\}
      \end{equation*}
      for fixed finite tuples $(v_i)$ of vectors (cf. \cite[\S 1]{zbMATH02152978});

    \item alternatively, those origin-symmetric convex bodies whose symmetry group is maximal compact in the general linear group of the ambient space (so a conjugate in $GL(\bR\text{ or }\bC)$ of the corresponding orthogonal/unitary group). 
    \end{itemize}
  \end{enumerate}  
\end{remarks}

The continuity of the map
\begin{equation*}
  \left(\text{$n$-dimensional convex body }K\subset \bR^n\right)
  \xmapsto{\quad}
  \left(\text{L\"owner ellipsoid }K^L(K)\right)
\end{equation*}
with respect to the Hausdorff distance on the space of convex bodies in $\bR^n$ is attributed by \cite[\S 2]{zbMATH02121437}, \cite[p.607]{zbMATH02120325}, \cite[\S 5]{MR1793468} and so on to \cite{zbMATH01614358}). Because
\begin{itemize}[wide]
\item on the one hand \cite{zbMATH01614358} has proven somewhat difficult to obtain;

\item on the other hand, the present context is somewhat broader, with ellipsoids respectively housed by continuously-varying fibers but not regarded as inhabiting a single ambient metric space that affords a Hausdorff metric (though this will not be much of an issue in first instance, given local triviality; cf. proof of \Cref{pr:contellips});

\item and finally, it is convenient to address the real and complex versions of the statement simultaneously,
\end{itemize}
we record in \Cref{pr:contellips} an auxiliary result used in passing in the proof of \Cref{th:isbddunit} above. {\it Vertical} subsets (of the total space) of a Banach bundle $\cE$ are those contained in a single fiber $\cE_x$ for $x\in \text{base space }X$. Recalling (\cite[p.14]{dg_ban-bdl}, \cite[Axiom V following Definition 3.3]{zbMATH03539905}) that {\it full} Banach bundles are those whose points all belong to images of global sections, we adopt some short-hand terminology for a weaker concept (\cite[p.61, Axiom IV]{zbMATH03539905} or \cite[\S 1.2, Existence Axiom II]{hk_shv-bdl}):

\begin{definition}\label{def:wfull}
  A Banach bundle $\cE\xrightarrowdbl{} X$ is {\it approximately (or weakly) full} if for every $p\in \cE_x$ and every $\varepsilon$ there is a local section defined in a neighborhood of $x\in X$ whose value at $x$ is $\varepsilon$-close to $p$. 
\end{definition}

\begin{proposition}\label{pr:contellips}
  For a weakly full $n$-homogeneous (F) Banach bundle $\cE\xrightarrowdbl{} X$ be over any space the maps
  \begin{equation*}
    \left(\text{vertical $n$-dimensional convex bodies}\right)
    =:
    \cK^n(\cE)
    \ni K    
    \xmapsto{\quad}
    \begin{cases}
      \text{L\"owner ellipsoid }
      K^L(K)
      \\
      \text{John ellipsoid }
      K_J(K)
      \\
    \end{cases}
    \in \cK^n(\cE)
  \end{equation*}
  are Vietoris-continuous.
\end{proposition}
\begin{proof}
  We treat the L\"owner case, the other claim being entirely parallel. 
  
  Consider
  \begin{equation*}
    X\ni
    x_{\lambda}
    \xrightarrow[\lambda]{\quad\text{convergent {\it net} \cite[Definition 11.2]{wil_top}}\quad}
    x
    \in X.
  \end{equation*}
  with $n$-dimensional convex bodies $K_{\lambda}\subset \cE_{x_{\lambda}}$ Vietoris-converging to $K\in \cK^n(\cE_x)$. The hypotheses suffice to trivialize $\cE$ locally around $x$: given a basis $(e_i)_{i=1}^n\subset \cE_x$, local sections around $x$ taking values respectively close (or indeed, equal \cite[Proposition 13.15]{fd_bdl-1}) to $e_i$ will be linearly independent (and hence a basis) in all fibers above points sufficiently close to $x$ (as argued in \cite[p.231, Remarque preceding \S 2]{dd}, say). It follows that we may as well assume everything in sight inhabits a single $n$-dimensional real or complex vector space $E:=\cE_x$.

  There is no harm in assuming $\{K_{\lambda}\}_{\lambda}$ relatively compact in the Hausdorff topology on compact subsets of $E$, whence also the relative compactness of $\{K^L(K_{\lambda})\}_{\lambda}$. Passing to a subnet if necessary, we can thus assume \cite[Theorem 17.4]{wil_top}
  \begin{equation*}
    K^L(K_{\lambda})
    \xrightarrow[\quad\lambda\quad]{}
    \text{ellipsoid }K'\text{ containing }K.
  \end{equation*}
  That $K'$ is exactly the L\"owner ellipsoid $K^L(K)$ then follows from the latter's characterization (\cite[Theorem 15.4]{tj_bm}, \cite[Theorem 10.12.1]{schn_cvx_2e_2014}, etc.): for every $\lambda$ there are
  \begin{equation}\label{eq:char.ellips}
    \begin{aligned}
      (v_{i,\lambda})_{i=1}^N
      &\subset
        \partial K_{\lambda}\cap\partial K^L(K_{\lambda})\\
      (c_{i,\lambda})_{i=1}^N
      &\subset
        \bR_{> 0}
    \end{aligned}    
    \quad\text{with}\quad
    \sum_i c_{i,\lambda} v_{i,\lambda}\otimes v_{i,\lambda}^{*,\lambda}
    =
    \id\in \End(V)\cong V\otimes V^*,
  \end{equation}
  where
  \begin{equation*}
    V\ni v
    \xmapsto[\quad\cong\quad]{\quad}
    v^{*,\lambda}\in V^*
  \end{equation*}
  is the isomorphism induced by the inner product $\braket{-\mid-}_{\lambda}$ whose underlying norm has $K^L(K_{\lambda})$ as its unit ball and
  \begin{equation*}
    \dim V\le N\le
    \begin{cases}
      \frac{\dim V\cdot (\dim V+1)}2&\text{over $\bR$}\\
      \dim^2 V&\text{over $\bC$}.
    \end{cases}
  \end{equation*}
  After once more passing to an appropriate subnet we can assume all $(v_{i,\lambda})_{\lambda}$ and $(c_{i,\lambda})_{\lambda}$ convergent (the $c_i$ sum up to $n$ \cite[(10.112)]{schn_cvx_2e_2014}, so the tuples $(c_i)_i$ range over a relatively compact subspace of $\bR^n$), and the conditions \Cref{eq:char.ellips} are closed under passing to limits. 
\end{proof}


\section{Expectations onto function spaces}\label{se:exp}

As a joint perusal of \cite{MR3056657,bg_cx-exp} makes clear, unitarizing Banach bundles (\cite[Problem 4.7]{MR3056657}) is intimately linked with the existence of finite-index conditional expectations $\Gamma_b(\cA)\xrightarrow{E}C(X)$ for continuous (unital) $C^*$-algebra bundles $\cA\xrightarrowdbl{}X$ (\Cref{qu:bg_pb3.11}).

\begin{remark}
  The relevant inequality in \Cref{qu:bg_pb3.11}\Cref{item:qu:bg_pb3.11:opt.ind} is $K(E)\le r(\cA)$, the opposite having been established in \cite[Theorem 1.4]{bg_cx-exp}. For that reason, we will occasionally refer to expectations $E$ as in \Cref{qu:bg_pb3.11} as {\it optimal}. The term applies equally to states on finite-dimensional $C^*$-algebras ({\it the} optimal state, for it is unique \cite[Lemma 3.2(ii)]{bg_cx-exp}), as does the notation $r(\cdot)$: simply regard such algebras as bundles over singletons. 
\end{remark}

The following construction answers \Cref{qu:bg_pb3.11} negatively.

\begin{example}\label{ex:bg_pb3.11_no}
  The base space of the (complex, unital $C^*$) bundle will be $X:=[-1,1]$, with $\cA\xrightarrowdbl{}X$ defined as follows.
  \begin{itemize}[wide]
  \item Over $U:=X\setminus\{0\}$ the bundle is trivial:
    \begin{equation*}
      \cA|_{U}\cong U\times M_3
      ,\quad
      U:=X\setminus\{0\} = [-1,0)\cup (0,1]. 
    \end{equation*}

  \item The exceptional fiber $\cA_0$ is 2-dimensional (so $\cA_x\cong \bC^2$).

  \item And the bundle is assembled via the usual gadget (\cite[\S 13.18]{fd_bdl-1}, \cite[Proposition 1.3]{dg_ban-bdl}, \cite[Proposition 3.6]{zbMATH03539905}) of singling out a space of sections whose set of images is dense in every fiber.

    In the present case, we will describe the {\it entire} section space $\Gamma_b(\cA)$ as
    \begin{equation}\label{eq:gammaa}
      \Gamma(\cA)
      =
      \Gamma_{\ell}\times_{\bC^2}\Gamma_r,
    \end{equation}
    where
    \begin{equation}\label{eq:gammas}
      \begin{aligned}
        \Gamma_{\ell}
        &:=
          \left\{[-1,0]\xrightarrow[\text{continuous}]{f}M_3\ |\ f(0)=\mathrm{diag}(a,a,b)\quad\text{for some}\quad a,b\in \bC \right\}\\
        \Gamma_{r}
        &:=
          \left\{[0,1]\xrightarrow[\text{continuous}]{f}M_3\ |\ f(0)=\mathrm{diag}(a,b,b)\quad\text{for some}\quad a,b\in \bC \right\},\\
      \end{aligned}
    \end{equation}
    and the maps $\Gamma_{\ell,r}\to \bC^2$ restrict functions at $0$ and identify the displayed diagonal matrices with $(a,b)\in \bC^2$. 
  \end{itemize}
  The claim is that this is a counterexample to the converse to \cite[Theorem 1.4]{bg_cx-exp}, thus providing a negative answer to \cite[Problem 3.11]{bg_cx-exp}: $r(\cA)=3$, but a conditional expectation $\Gamma(\cA)\xrightarrow{E}C(X)$ with positive $3E-\id$ would have \cite[Lemma 3.2(ii)]{bg_cx-exp} to be the normalized trace over the non-degenerate locus, and hence cannot extend across the exceptional fiber at $0$.
\end{example}

\begin{remarks}\label{res:gluemods}
  \begin{enumerate}[(1),wide]
  \item Specifying the $C(X)$-module $\Gamma(\cA)$ gives a complete characterization of the bundle via the correspondence \cite[Scholium 6.7]{hk_shv-bdl} between bundles and {\it locally convex} \cite[\S 6.1]{hk_shv-bdl} Banach $C(X)$-modules for compact Hausdorff $X$. Local convexity is in this case an easy check. 

  \item The construction of $\Gamma(\cA)$ as a pullback \Cref{eq:gammaa} is a familiar one \cite[p.20, Basic construction]{miln_k}: we have a pullback representation
    \begin{equation}\label{eq:rng.plb}
      \begin{tikzpicture}[>=stealth,auto,baseline=(current  bounding  box.center)]
        \path[anchor=base] 
        (0,0) node (l) {$C([-1,0])$}
        +(3,.5) node (u) {$C([-1,1])$}
        +(3,-.5) node (d) {$\bC\cong C(\{0\})$}
        +(6,0) node (r) {$C([0,1])$}
        ;
        \draw[<<-] (l) to[bend left=6] node[pos=.5,auto] {$\scriptstyle $} (u);
        \draw[->>] (u) to[bend left=6] node[pos=.5,auto] {$\scriptstyle $} (r);
        \draw[->>] (l) to[bend right=6] node[pos=.5,auto,swap] {$\scriptstyle $} (d);
        \draw[<<-] (d) to[bend right=6] node[pos=.5,auto,swap] {$\scriptstyle $} (r);
      \end{tikzpicture}
    \end{equation}
    for the ring $C(X)$, with all maps denoting restrictions. The spaces $\Gamma_{\ell}$ and $\Gamma_r$ of \Cref{eq:gammas} are then section spaces for bundles over $X_{\ell}:=[-1,0]$ and $X_r:=[0,1]$ respectively, and hence modules over the lateral rings in \Cref{eq:rng.plb}. Restriction of sections at $0$ can be identified with tensoring
    \begin{equation}\label{eq:gamma.tenss}
      \Gamma_{\ell}\otimes_{C(X_{\ell})}\bC
      \quad\text{and}\quad
      \Gamma_{r}\otimes_{C(X_{r})}\bC
    \end{equation}
    respectively along the two bottom arrows of \Cref{eq:rng.plb}, and \Cref{eq:gammas} describes an isomorphism between \Cref{eq:gamma.tenss}; that isomorphism provides the ``gluing'' information necessary in order to piece together a single module over the pullback ring $C(X)$ via the aforementioned \cite[p.20]{miln_k}.

  \item \Cref{ex:bg_pb3.11_no} is not unrelated to the phenomenon of {\it pushouts} \cite[Definition 11.30]{ahs}
    \begin{equation*}
      \begin{tikzpicture}[>=stealth,auto,baseline=(current  bounding  box.center)]
        \path[anchor=base] 
        (0,0) node (l) {$A$}
        +(2,.5) node (u) {$D$}
        +(2,-.5) node (d) {$C$}
        +(4,0) node (r) {$B$}
        ;
        \draw[<-left hook] (l) to[bend left=6] node[pos=.5,auto] {$\scriptstyle $} (u);
        \draw[right hook->] (u) to[bend left=6] node[pos=.5,auto] {$\scriptstyle $} (r);
        \draw[->] (l) to[bend right=6] node[pos=.5,auto,swap] {$\scriptstyle $} (d);
        \draw[<-] (d) to[bend right=6] node[pos=.5,auto,swap] {$\scriptstyle $} (r);
      \end{tikzpicture}
    \end{equation*}
    of finite-dimensional $C^*$-algebras $A$, $B$ and $D$ failing to be {\it residually finite-dimensional (RFD)} \cite[Definition V.2.1.10]{blk}: $C$ fails to embed into a $C^*$ product of finite-dimensional $C^*$-algebras precisely \cite[Theorem 4.2]{adel_full-amalg} when $A$ and $B$ admit faithful traces agreeing on $D$. The non-existence of such traces, for the two embeddings
    \begin{equation}\label{eq:c2.2.m3}
      D\cong \bC^2
      \ni(a,b)
      \xmapsto{\quad}
      \mathrm{diag}(a,a,b)
      \quad\text{or}\quad
      \mathrm{diag}(a,b,b)
      \in (A,B\cong M_3),
    \end{equation}
    is precisely what drove \Cref{ex:bg_pb3.11_no} (cf. \cite[Example 1]{zbMATH06233935}, to the same effect). We formalize the present observation in \Cref{le:plbk.cones} below. 
  \end{enumerate}
\end{remarks}

\Cref{le:plbk.cones} refers to the {\it mapping cylinder} \cite[Problem 6.M]{wo} $Z_{\iota}$ of a $C^*$-morphism $D\xrightarrow{\iota} A$: the pullback
\begin{equation*}
  \begin{tikzpicture}[>=stealth,auto,baseline=(current  bounding  box.center)]
    \path[anchor=base] 
    (0,0) node (l) {$D$}
    +(2,.5) node (u) {$Z_{\iota}$}
    +(2,-.5) node (d) {$A$}
    +(4,0) node (r) {$C([0,1],A)$.}
    ;
    \draw[<-] (l) to[bend left=6] node[pos=.5,auto] {$\scriptstyle $} (u);
    \draw[->] (u) to[bend left=6] node[pos=.5,auto] {$\scriptstyle $} (r);
    \draw[->] (l) to[bend right=6] node[pos=.5,auto,swap] {$\scriptstyle \iota$} (d);
    \draw[<<-] (d) to[bend right=6] node[pos=.5,auto,swap] {$\scriptstyle \mathrm{eval}_0$} (r);
  \end{tikzpicture}
\end{equation*}

\begin{lemma}\label{le:plbk.cones}
  Let $D\lhook\joinrel\xrightarrow{\iota_i}A_i$, $i=0,1$ be unital embeddings of finite-dimensional $C^*$-algebras, and write
  \begin{itemize}[wide]
  \item $\cA_{\iota_0,\iota_1}\xrightarrowdbl{}X$ for the bundle over 
    \begin{equation}\label{eq:xpush}
      X
      :=
      \bigg(\text{pushout of two copies of $0\lhook\joinrel\xrightarrow{\quad}[0,1]$}\bigg)
      \cong
      [-1,1]
    \end{equation}
    whose space of sections is the pullback $Z_{\iota_0}\times_D Z_{\iota_1}$ of the two maps $Z_{\iota_i}\to D$;

  \item and $\cA_{\iota_i}$, $i=0,1$ for the restrictions of $\cA_{\iota_0,\iota_1}$ to the two copies of $[0,1]$ in \Cref{eq:xpush}.
  \end{itemize}

  If the ranks $r(\cA_{\iota_i})$ are equal, there is an optimal expectation $\Gamma(\cA_{\iota_0,\iota_1})\xrightarrow{E}X$ if and only if the unique optimal tracial states on $A_i$, $i=0,1$ agree on $D$.  \qedhere
\end{lemma}

\begin{remark}\label{re:even.ab}
  Note incidentally that \Cref{le:plbk.cones} applies to bundles of {\it commutative} $C^*$-algebras: \Cref{eq:c2.2.m3} can easily be interpreted as morphisms to $\bC^3$ rather than $M_3$. The conclusion is that for the corresponding {\it branched cover} \cite[\S 1]{pt_brnch} $Y\to X:=[-1,1]$ with cardinality-3 generic fibers (over $[-1,1]\setminus\{0\}$) and exceptional fiber of cardinality $2$ at $0\in X$ there is no expectation $C(Y)\xrightarrow{E}C(X)$ with $K(E)=3$. There are \cite[Theorem 1.1]{pt_brnch}, of course, conditional expectations with {\it larger} $K(\cdot)$. 
\end{remark}

The principle underlying \Cref{le:plbk.cones} is broader than that statement suggests. The gadgetry of {\it Bratteli diagrams}, familiar (\cite[Chapter 2]{effr_dim}, \cite[\S XIX.1]{tak3}) as bookkeeping tools in studying embeddings of finite-dimensional $C^*$-algebras, will be useful in phrasing the generalization in \Cref{pr:same.brat}. 

Following those sources, we extend the usual notation $M_n$ to {\it multi-matrix} algebras
\begin{equation*}
  M_{\bf n}
  :=
  M_{n_1}\times\cdots\times M_{n_k}
  ,\quad
  {\bf n}:=(n_1,\ \cdots,\ n_k)
  \in
  \bZ_{>0}^{\ell({\bf n}):=k}.
\end{equation*}
An embedding $M_{\bf m}\lhook\joinrel\xrightarrow{\iota} M_{\bf n}$ (unital for us, here, but not necessarily so in the literature) can be described uniquely up to {\it (inner) equivalence}, i.e. \cite[p.7]{effr_dim} conjugation by (inner) automorphisms on both sides by either
\begin{itemize}[wide]
\item the $\ell({\bf n})\times\ell({\bf m})$ matrix $T=T_{\iota}$ defined by
  \begin{equation*}
    \begin{aligned}
      T_{ij}
      &:=
        \mathrm{Tr}_i(\iota(\text{minimal projection in factor }M_{m_j}))
      \quad\text{where}\\
      \mathrm{Tr}_i
      &:=
        \text{un-normalized trace of the factor }M_{n_i};
    \end{aligned}
  \end{equation*}
\item or the {\it Bratteli diagram} of $\iota$, a bipartite graph on the bipartition
  \begin{equation*}
    \left\{1,\ \cdots,\ \ell(m)\right\}
    \sqcup
    \left\{1,\ \cdots,\ \ell(n)\right\}
  \end{equation*}
  with $T_{ij}$ edges connecting vertex $j$ on the left-hand side and vertex $i$ on the right. 
\end{itemize}
We may as well refer to the former construct as a {\it Bratteli matrix} for the embedding, but we use the two notions interchangeably in any case: they encode the same information. 

\begin{definition}\label{def:brat.germ}
  Let $\cA\xrightarrowdbl{}X$ be a unital subhomogeneous (F) $C^*$ bundle over a locally paracompact $T_2$ space (so that in particular the bundle is full \cite[Proposition 3.4]{zbMATH03539905}).

  A {\it germ at $x\in X$} of invariants attached to embeddings of $C^*$-algebras (such as Bratteli diagrams or matrices, or other invariants derived therefrom: row/column sums, etc.) is an equivalence class of such invariants obtained as follows:
  \begin{itemize}[wide]
  \item extend the fiber $\cA_x$ to a $(\dim \cA_x)$-homogeneous $C^*$ bundle $\cB\xrightarrowdbl{}U$ over a neighborhood $U\ni x$ (possible essentially by \cite[Theorem 3.1]{fell_struct}, but see also \Cref{pr:ss.fibers.extend}: the local {\it compactness} of the base space in the former result can easily be slackened to local paracompactness);

  \item consider all Bratteli diagrams/matrices attached to embeddings $\cB_y\le \cA_y$, $y\in U$ (and the invariants they generate, whatever those may be);

  \item identify those mutually conjugate under (possibly outer) automorphisms of $\cA_y$;
    
  \item and retain only those achievable over arbitrarily small neighborhoods $U\ni x$. 
  \end{itemize}
\end{definition}

\begin{remark}\label{re:why.germ}
  The fact that we are interested only in diagrams which persist arbitrarily close to $x$ is what justifies the {\it germ} terminology, familiar from sheaf theory \cite[\S I.1, p.2]{bred_shf_2e_1997}: germs of functions are equivalence classes thereof, identifiable if equal across a neighborhood of the base point in question. 
\end{remark}


We record the following result on extending finite-dimensional semisimple Banach algebras to nearby fibers. It is very much in the spirit of \cite[Theorem 3.1]{fell_struct}, which covers the $C^*$ case over locally compact Hausdorff base spaces. 

\begin{proposition}\label{pr:ss.fibers.extend}
  Let $\cA\xrightarrowdbl{}X$ be a unital Banach-algebra (F) bundle over a locally paracompact Hausdorff base space.

  \begin{enumerate}[(1),wide]
  \item A finite-dimensional semisimple subalgebra $A\le \cA_x$ extends to a $(\dim A)$-homogeneous Banach-algebra bundle locally around $x$.

  \item The analogous result holds {\it mutatis mutandis} for $C^*$-algebra bundles.  
  \end{enumerate}   
\end{proposition}
\begin{proof}
  Pick a basis for $A$ and extend it locally around $x$ to linearly-independent sections $(s_i)_i$, as allowed by the fullness \cite[Proposition 3.4]{zbMATH03539905} of the bundle. The spaces $\mathrm{span}\{s_i(x')\}_i$ of course constitute a trivial Banach-{\it space} bundle locally around $x$, but the maps
  \begin{equation*}
    \cA_x
    \supseteq
    A
    \ni
    \sum_{i}c_i s_i(x)
    \xmapsto{\quad\varphi_{x'}\quad}
    \sum_{i}c_i s_i(x')
    \in \cA_{x'}
  \end{equation*}
  will not, generally, be multiplicative. For $x'$ sufficiently close to $x$ however, they are uniformly bounded and have uniformly bounded discrepancies from multiplicativity:
  \begin{equation*}
    \exists\left(
      \text{nbhd }U\ni x
      ,\quad
      K,\delta>0
    \right)
    \left(
      \forall x'\in U
      \ :\
      \|\varphi_{x'}\|\le K
      \ \text{and}\
      \left\|\varphi_{x'}^{\vee}\right\|\le \delta
    \right)
  \end{equation*}
  where
  \begin{equation*}
    \varphi^{\vee}(a,b)
    :=
    \varphi(ab)-\varphi(a)\varphi(b)
    \quad\left(\text{see \cite[\S 1]{john_approx}}\right).
  \end{equation*}
  Semisimplicity and finite-dimensionality then allow us to deform $\varphi_{x'}$ continuously into morphisms by \cite[Corollary 3.2]{john_approx} (and the proof of \cite[Theorem 3.1]{john_approx}, which it in turn relies on).
\end{proof}

\begin{remark}\label{re:not.h}
  \cite[Proposition 3.4]{zbMATH03539905} proves fullness in the broader context of (H) bundles, but \Cref{pr:ss.fibers.extend} plainly would not hold in that generality: $\cA$ might have a single non-zero fiber $\cA_x$ and vanish elsewhere. Absent norm \emph{continuity} (rather than only upper semicontinuity) neither the linear independence of the extended sections nor the subsequent norm estimates can be taken for granted away from $x$ (in no matter how small a neighborhood thereof).
\end{remark}


The argument supporting \Cref{ex:bg_pb3.11_no} (and \Cref{le:plbk.cones}) in fact proves:

\begin{proposition}\label{pr:same.brat}
  Let $\cA\xrightarrowdbl{}X$ be a unital subhomogeneous (F) $C^*$ bundle over a compact Hausdorff space, and suppose the (automatically open) subset
  \begin{equation*}
    \left\{x\in X\ |\ r(\cA_x) = r(\cA)\right\}
    \subseteq
    X
  \end{equation*}
  is dense.

  The existence of an optimal expectation $\Gamma(\cA)\xrightarrow{E}C(X)$ is then equivalent to the condition that $\cA$ have unique germs (in the sense of \Cref{def:brat.germ}) of tuples
  \begin{equation}\label{eq:col.sum.tup}
    \left(\sum_i T_{ij}\right)_j
    =
    \left(\text{sum along column $j$}\right)_j
    ,\quad (T_{ij})_{i,j}=\text{Bratteli matrix}
  \end{equation}
  at every $x\in X$.
\end{proposition}
\begin{proof}
  As in \Cref{ex:bg_pb3.11_no}, an expectation $\Gamma(\cA)\xrightarrow{E}C(X)$ with $K(E)=r:=r(\cA)$ would have to induce on all fibers $\cA_x$ with $r(\cA_x)=r(\cA)$ that optimal state. 

  Conversely, that choice provides an expectation $E|_U$ with appropriate $K$-constant, with
  \begin{equation}\label{eq:large.r.loc}
    U:=\left\{x\in X\ |\ r(\cA_x)=r\right\},
  \end{equation}
  assumed dense in $X$. That unique $E|_U$ extends across all of $X$ precisely when, for every $x\in X$, the restrictions of the optimal states on $\cA_{x'}$ to $\cA_x$ along embeddings attached to Bratteli germs are all equal. Given the expression \cite[(3.5)]{bg_cx-exp}
  \begin{equation}\label{eq:opt.state}
    \sum_{i=1}^{\ell(n)} \frac {n_i}{|{\bf n}|}\left(\text{normalized trace on }M_{n_i}\right)
    ,\quad
    |{\bf n}|:=\sum_i n_i
    ,\quad
    {\bf n} = (n_i)_i
  \end{equation}
  for the optimal state on the multi-matrix algebra $M_{\bf n}$, this in turn translates to the hypothesis in the statement: if $(T_{ij})_{i,j}$ is the $\ell({\bf n})\times \ell({\bf m})$ Bratteli matrix of a unital embedding $M_{\bf m}\lhook\joinrel\to M_{\bf n}$, then the restriction of the optimal state \Cref{eq:opt.state} on $M_{\bf n}$ along that embedding is
  \begin{equation*}
    \sum_{j=1}^{\ell(m)}
    \left(
      \frac{\sum_{i}T_{ij}}{|\bf n|}
      \cdot
      \left(\text{un-normalized trace on }M_{m_j}\right)
    \right),
  \end{equation*}
  so that \Cref{eq:col.sum.tup} is a complete invariant for that restriction. 
\end{proof}

\begin{remark}\label{re:get.bg_prop3.7}
  \Cref{pr:same.brat} recovers \cite[Proposition 3.7]{bg_cx-exp}, asserting the existence of an expectation with $K(E)=2$ when that is the rank:
  \begin{itemize}[wide]
  \item Embeddings among the algebras $\bC$, $\bC^2$ and $M_2$ of rank $\le 2$ are unique up to inner automorphism, so the Bratteli germs themselves are in that case unique.

  \item And whether or not the set $U$ of \Cref{eq:large.r.loc} is dense is irrelevant here, for there is no problem in extending traces across the locus where $r(\cA_x)=1$ (and hence $\cA_x\cong \bC$). 
  \end{itemize}
\end{remark}

The existence \cite[Proposition 3.4]{bg_cx-exp} of optimal expectations for {\it homogeneous} bundles of course also follows from \Cref{pr:same.brat}. More generally, the result holds in the ``multiplicity-free'' case: 

\begin{corollary}\label{cor:mult.free}
  Let $\cA\xrightarrowdbl{}X$ be a subhomogeneous (F) $C^*$ bundle over a compact Hausdorff space, with all Bratteli-matrix germs $(T_{ij})_{i,j}$ having singleton columns (one entry equal to 1 and 0s elsewhere).

  The bundle then admits an optimal expectation $\Gamma(\cA)\to C(X)$.
\end{corollary}
\begin{proof}
  The germ-uniqueness constraint of \Cref{pr:same.brat} is satisfied, as is the density requirement: the statement's condition on the matrix germs is equivalent with $r(\cA_x)$ {\it all} being equal (to each other, hence also to $r(\cA)=\sup_x r(\cA_x)$).
\end{proof}


\begin{theorem}\label{th:nc.brnch.cov.cpctmetr}
  For any unital subhomogeneous (F) $C^*$ bundle $\cA\xrightarrowdbl{}X$ over compact metrizable $X$ there exists a finite-index expectation $\Gamma(\cA)\xrightarrow{E}C(X)$. 
\end{theorem}
  
The proof requires some background on various notions of semicontinuity for set-valued maps into (in this case {\it locally convex} \cite[\S 18.1]{k_tvs-1} linear) topological spaces. Recall that a map $X\xrightarrow{\cK_{\bullet}} 2^Y$ for topological spaces $X$ and $Y$ is {\it lower semicontinuous (LSC)} \cite[Definition 15.1]{gorn_fp} (sometimes \cite[Definition 17.2]{ab_inf-dim-an} {\it hemi}continuous) if 
\begin{equation}\label{eq:lsc}
  \forall \text{ open }W\subseteq Y
  ,\quad
  \left\{x\in X\ |\ \cK_x\cap W\ne \emptyset\right\}\subseteq X\text{ is open}.
\end{equation}
Such conditions play an essential role in Michael-type {\it selection theorems} (\cite[Theorem 1.1]{zbMATH06329568}, or the original \cite[Theorem 3.2'']{mich_contsel-1}), one of which we will employ shortly.

\pf{th:nc.brnch.cov.cpctmetr}
\begin{th:nc.brnch.cov.cpctmetr}
  The idea is to prove the existence of a judiciously-chosen map
  \begin{equation}\label{eq:phi2k}
    X\ni x
    \xmapsto{\quad \cK_{\bullet}\quad}
    \text{closed convex }\cK_x
    \subset
    \cS_{++}(\cA_x)
    :=
    \left\{\text{faithful state-space of }\cA_x\right\},
  \end{equation}
  with the right-hand sides regarded as subspaces of the common Banach space $\Gamma(\cA)^*$ via the embedding $\cA_x^*\le \Gamma(\cA)^*$ dual to the surjection $\Gamma(\cA)\xrightarrowdbl{}\cA_x$ (analogous to) \Cref{eq:fib.is.quot}. We will require the following of $\cK_{\bullet}$:
  \begin{enumerate}[(a),wide]
  \item\label{item:bddk} every $\cK_x$ consists not only of faithful states, but in fact of states with $K$-constant dominated by a uniform (in $x\in X$) upper bound $C\ge 1$;
    
  \item\label{item:semicont} and $\cK_{\bullet}$ is weak$^*$ LSC.
  \end{enumerate}

  We portion out the rest of the argument. 
  
  \begin{enumerate}[(I),wide]
  \item\label{item:th:nc.brnch.cov.tfg:ass.cond} {\bf Conclusion, assuming \Cref{item:bddk,item:semicont}.} When $X$ is compact metrizable $\Gamma(\cA)$ is separable by \cite[Example 19.5(iii)]{gierz_bdls} (along with \cite[Proposition 16.4]{gierz_bdls}, providing the $T_2$ property for the total space $\cA$ required by that example). The weak$^*$-topologized unit ball $\Gamma(\cA)^*_1\subset \Gamma(\cA)^*$, housing the compact convex sets $\cK_{\bullet}$ of \Cref{eq:phi2k},
    \begin{itemize}[wide]
    \item is a {\it uniform convex space} in the sense of \cite[Definition 2.2]{horv_top-cvx} (as noted in \cite[sentence preceding Proposition 2.1]{horv_top-cvx}, being a convex subset of the locally convex topological vector space);

    \item is metrizable \cite[Theorem V.5.1]{conw_fa} by separability and compact in any case, regardless of separability, by {\it Alaoglu's theorem} \cite[Theorem V.3.1]{conw_fa}), so in fact {\it completely} metrizable;

    \item and obviously has homotopically trivial (indeed, {\it contractible} \cite[Definition 32.6]{wil_top}) {\it polytopes}, i.e. \cite[\S 1]{horv_top-cvx} convex hulls of finite subsets. 
    \end{itemize}    
    We thus meet the hypotheses of the Michael-type selection theorem \cite[Theorem 3.4]{horv_top-cvx} applicable in this setup, so there is a continuous selection
    \begin{equation*}
      X\ni x
      \xmapsto{\quad}
      E_x\in \cK_x
      \xRightarrow{\quad}
      K(\cE_x)\le C,\quad\forall x\in X
    \end{equation*}
    for $\cK_{\bullet}$, aggregating into the desired expectation $\Gamma(\cA)\xrightarrow{E}C(X)$ with $K$-constant $\le C$.

  \item\label{item:th:nc.brnch.cov.tfg:constr} {\bf Construction of $\cK_{\bullet}$.} We begin by first placing the (finitely many) fiber isomorphism classes in recursively-defined classes $\cC_k$, $k\in\bZ_{\ge 0}$ (finitely many such classes; the construction stabilizes):
    \begin{itemize}[wide]
    \item $\cC_0$ consists of those (isomorphism types of) $\cA_x$ which contain no others (strictly) in germs of embeddings attached to $\cA$. These are, in a sense, minimal (e.g. the exceptional fiber $\bC^2$ in \Cref{ex:bg_pb3.11_no}).
      
    \item For the recursion step, having defined $\cC_\bullet$, $0\le \bullet\le k-1$, place in $\cC_{k}$ those fiber isomorphism types $B$ featuring in embedding germs
      \begin{equation*}
        B'\le B
        ,\quad
        B'\in \cC_{k-1}.
      \end{equation*}
    \end{itemize}
    The definition of $\cK_{\bullet}$ will also be recursive, each step extending the definition across a larger-index
    \begin{equation*}
      X_k:=\left\{x\in X\ |\ \cA_x\in \cC_k\right\}. 
    \end{equation*}
    For $x\in X_0$ take for the compact convex set $\cK_x$ the singleton consisting of the optimal state on $\cA_x$ (the uniqueness of that state renders it $\Aut(\cA_x)$-invariant).
    
    Next, for $x\in X_1$ consider (the finitely many) embedding germs $\cA_y\le \cA_x$ with $y\in X_0$, extend the elements of $\cK_y$ to faithful states on $\cA_x$ across those embeddings, and form the $\Aut(\cA_x)$-invariant convex hull of that space of extensions. Neither operation will affect faithfulness, so that the result will be a compact convex subset
    \begin{equation*}
      \cK_x\subset \cS_{++}(\cA_x)
      ,\quad
      \sup\left\{K(\theta)\ \bigg|\ \theta\in \bigcup_{x\in X_1}\cK_x\right\}\le
      \text{some }C_1<\infty.
    \end{equation*}
    Continue the procedure, producing $\cK_x$ for $x\in X_k$ by extending elements of $\cK_y$ for embedding germs $\cA_y\le \cA_x$ with $y\in X_{k-1}$, so that 
    \begin{equation*}
      \sup\left\{K(\theta)\ \bigg|\ \theta\in \bigcup_{x\in X_k}\cK_x\right\}\le
      \text{some }C_k<\infty.
    \end{equation*}

  \item\label{item:th:nc.brnch.cov.tfg:check.cond} {\bf Verifying \Cref{item:bddk,item:semicont}.} The first requirement \Cref{item:bddk} we made of $x\mapsto \cK_x$ holds with $C:=\sup_k C_k$. As to \Cref{item:semicont}, fix first a
    \begin{equation*}
      x_{\lambda}
      \xrightarrow[\lambda]{\quad\text{convergent net}\quad}
      x
      \quad\text{in}\quad
      X.
    \end{equation*}
    By the very construction of the $\cK_{\bullet}$, the states in $\cK_{x_{\lambda}}$, when restricted to $\cA_x$ along embedding germs $\cA_x\le \cA_{x_{\lambda}}$, recover all states in $\cK_x$:
    \begin{equation*}
      \text{for large $\lambda$}
      ,\quad
      \cK_{x_{\lambda}}|_{\cA_x}\supseteq \cK_x.
    \end{equation*}
    But then, for any fixed state $\varphi\in \cK_x$ on $\cA_x$ and finite collection of sections $s_i\in \Gamma(\cA)$, we can find extensions $\varphi_{\lambda}\in \cA_{x_{\lambda}}$ whose values $\varphi_{\lambda}(s_i)$ respectively approach (with increasing $\lambda$) $\varphi(s_i)$. This is sufficient for the required weak$^*$ LSC. 
  \end{enumerate}
  This concludes the proof. 
\end{th:nc.brnch.cov.cpctmetr}

\addcontentsline{toc}{section}{References}

\Addresses

\end{document}